\documentclass[11pt]{article}

\usepackage[margin=1in]{geometry}
\usepackage{amsmath,amssymb,amsfonts,amsthm}
\usepackage[colorlinks=true,linkcolor=blue,citecolor=blue,urlcolor=blue,pdfborder={0 0 0}]{hyperref}
\usepackage{titling}
\usepackage{booktabs}

\pretitle{\begin{flushleft}\Large}
\posttitle{\par\end{flushleft}\vskip 0.5em}

\theoremstyle{plain}
\newtheorem{theorem}{Theorem}[section]
\newtheorem{proposition}[theorem]{Proposition}
\newtheorem{lemma}[theorem]{Lemma}

\newtheorem{conjecture}[theorem]{Conjecture}
\newtheorem{question}[theorem]{Question}

\theoremstyle{definition}

\theoremstyle{remark}
\newtheorem{remark}[theorem]{Remark}

\newcommand{\C}{\mathbb{C}}

\preauthor{\begin{flushleft}}
\postauthor{\par\vspace{0.5em}\footnotesize
The Division of Physics, Mathematics and Astronomy, California Institute of Technology\\
Email: \href{mailto:looi@caltech.edu}{looi@caltech.edu}
\par\end{flushleft}}

\predate{\begin{flushleft}}
\postdate{\par\end{flushleft}}

\title{\bf Positive Berezin liminf does not imply essential positivity for radial Toeplitz operators on Bergman and Fock spaces}
\author{{\bf Sam Looi}}
\date{\small\today}

\begin{document}
\maketitle

\begin{abstract}
We study whether essential positivity
\[
\sigma_{\mathrm{ess}}(T_f)\subset [0,\infty)
\]
of a radial Toeplitz operator on Bergman and Fock spaces can be detected from the asymptotic behavior of its Berezin transform. For bounded real-valued radial symbols on $A^2(\mathbb{D})$, Per\"al\"a and Virtanen conjectured that
\[
\liminf_{|z|\to 1^-} \widetilde{f}(z)\ge 0
\]
should be equivalent to essential positivity, and they asked the analogous question on Fock space. Such a criterion would turn a spectral question into a scalar asymptotic test.

We prove that this fails even in the radial setting in which the conjecture was formulated. For every complex dimension $d\ge 1$, we construct explicit bounded real-valued radial symbols on the Bergman spaces $A^2(\mathbb{B}_d)$ and the Fock spaces $F^2(\mathbb{C}^d)$ whose Berezin transform has strictly positive limit inferior at the boundary (respectively, at infinity), while the essential spectrum of the corresponding Toeplitz operator contains a negative point. In particular, this disproves the Per\"al\"a--Virtanen conjecture in its original one-dimensional Bergman form and answers the analogous Fock-space question negatively; more generally, the radial Berezin liminf criterion fails in all dimensions in both settings. The underlying reason is that, for radial symbols, the Toeplitz eigenvalue sequence and the Berezin transform are different asymptotic averages of the same oscillatory symbol, and these averages damp the oscillation by different amounts.
\end{abstract}

\vspace{0.5em}
{\footnotesize\noindent\textit{2020 Mathematics Subject Classification.} 47B35 (primary); 30H20, 47A10, 41A60, 47B65 (secondary).}

\section{Introduction}

Perälä and Virtanen \cite{PeralaVirtanen} introduced the notion of essential positivity for self-adjoint operators, meaning that
\[
\sigma_{\mathrm{ess}}(T)\subset [0,\infty),
\]
and studied it for Toeplitz operators on the Hardy and Bergman spaces. On the Bergman space $A^2(\mathbb D)$ they proved that if $\mu$ is a real-valued radial Borel measure, $|\mu|$ is a Carleson measure for $A^2(\mathbb D)$, and the Berezin transform $\widetilde \mu$ has a boundary limit, then $T_\mu$ is essentially positive if and only if that limit is nonnegative. They then formulated the following conjecture and asked the corresponding question on the Fock space.

\begin{conjecture}[Per\"al\"a--Virtanen {\cite[Conjecture 1]{PeralaVirtanen}}]
\label{conj:PV-Bergman}
Let $f\in L^\infty(\mathbb D)$ be real-valued and radial. Then $T_f : A^2(\mathbb D) \to A^2(\mathbb D)$ is essentially positive if and only if
\[
\liminf_{|z|\to1^-}\widetilde f(z)\ge0.
\]
\end{conjecture}

\begin{question}[Per\"al\"a--Virtanen {\cite[Remark 12]{PeralaVirtanen}}]
\label{q:PV-Fock}
Let $f\in L^\infty(\mathbb C)$ be real-valued and radial. Is $T_f : F^2(\mathbb C) \to F^2(\mathbb C)$ essentially positive if and only if
\[
\liminf_{|z|\to\infty}\widetilde f(z)\ge0?
\]
\end{question}

At a conceptual level, this asks whether the sign of the essential spectrum of a radial Toeplitz operator can be read off from the asymptotic behavior of its Berezin transform. The boundary-limit theorem in \cite{PeralaVirtanen} is proved by a Tauberian argument, following Korenblum and Zhu \cite{KorenblumZhu}, and it depends crucially on the existence of an actual boundary limit for the Berezin transform. The conjectural $\liminf$ criterion is subtler. As we will show, this criterion compares two asymptotic procedures performed on the same radial symbol.

The main result of this paper is that the answer to both Conjecture~\ref{conj:PV-Bergman} and Question~\ref{q:PV-Fock} is ``no'' in the radial class in which the conjecture was formulated, not only in dimension one, but actually in all positive dimensions. For every $d\ge 1$ we construct explicit bounded real-valued radial symbols on the Bergman space $A^2(\mathbb B_d)$ and the Fock space $F^2(\mathbb C^d)$ such that the Berezin transform has strictly positive limit inferior, while the essential spectrum of the corresponding Toeplitz operator contains a negative point. Thus positive Berezin liminf does not characterize essential positivity for radial Toeplitz operators in any complex dimension, on either family of spaces, and positivity of $\liminf \widetilde f$ is stronger than nonnegativity.

Without radiality, Fulsche \cite{FulscheEssential} proved on the Fock space that, for bounded uniformly continuous symbols, such a Berezin criterion would imply an essential norm estimate contradicting a known obstruction adapted from Coburn. He further observed in \cite[Remark 2.12]{FulscheEssential} that the same indirect argument should transfer to the Bergman space, again without radiality. What remained open was whether radiality, which was the setting considered in the Perälä--Virtanen conjecture, might impose additional structure and restore the criterion. The present paper shows that it does not.

Our counterexamples are direct and completely explicit. On the Fock space, the symbol
\[
f(z)=\frac12+\cos(2|z|)
\]
is a counterexample on $F^2(\mathbb C^d)$ for every $n\ge1$. On the Bergman space, the one-dimensional counterexample is
\[
f(z)=\frac12+\cos\!\left(\log\frac{1}{1-|z|^2}\right),
\]
and in higher dimensions we use the same oscillatory profile with an explicit dimension-dependent offset
\[
f_d(z)=c_d+\cos\!\left(\log\frac{1}{1-|z|^2}\right).
\]
In each case the Berezin transform has positive limit inferior at the boundary (or at infinity in the Fock case), while the essential spectrum of the corresponding Toeplitz operator contains a negative point.
A noteworthy feature is that the Fock example is dimension-free, whereas on the Bergman side the fixed offset $1/2$ works only in low dimension; from dimension $12$ onward one must adjust the constant. Thus the higher-dimensional Bergman result is not merely a formal repetition of the one-dimensional argument.

The ideas of the proofs are as follows. For radial symbols on $F^2(\mathbb C^d)$ and $A^2(\mathbb B_d)$, the Toeplitz operator acts diagonally on monomials, and the corresponding eigenvalue depends only on the total degree. As a result, essential positivity is decided by the tail of one scalar eigenvalue sequence. The Berezin transform, however, averages the same radial profile in a different asymptotic regime. 
These two averages occur at different asymptotic scales. In the Bergman setting, the degree-$m$ eigenvalue averages the radial profile at boundary depth comparable to $(m+d)^{-1}$, whereas $\widetilde f(z)$ averages it at boundary depth comparable to $1-|z|^2$. In the Fock setting, the degree-$m$ eigenvalue averages the radial profile against the weight $r^{2m+2d-1}e^{-r^2}\,dr$, which is concentrated near $r=\sqrt{m+d}$, whereas $\widetilde f(z)$ averages against a unit-scale Gaussian centered at $z$, so for large $|z|$ it is concentrated near $r=|z|$. For the model symbol $f(r)=\frac12+\cos(2r)$, this means that the eigenvalue sequence reads the oscillation at phase $2\sqrt{m+d}$, while the Berezin transform reads it at phase $2|z|$.
For oscillatory radial symbols, this scale mismatch attenuates the oscillation by different factors. That is exactly what allows $\liminf \widetilde f$ to remain positive even though the essential spectrum still meets $(-\infty,0)$. In particular, the paper identifies not only that the Berezin criterion fails, but the asymptotic mechanism responsible for its failure. Despite the higher-dimensional statements, all of the asymptotic analysis reduces after radial reduction to explicit one-dimensional oscillatory integrals.

Radial Toeplitz operators on Bergman and Fock spaces have long been a natural testing ground precisely because of this diagonal structure; on the Bergman space see \cite{KorenblumZhu,GrudskyVasilevskiBergman}, and on the Fock space see \cite{GrudskyVasilevskiFock,EsmeralMaximenko,DewageOlafsson}. For Toeplitz operators on the Fock space with general bounded or measure symbols, boundedness, compactness, Toeplitz algebras, and Fredholmness were studied in \cite{IsralowitzZhu,BauerIsralowitz,FulscheCorrespondence,FulscheHagger}. Related formulas for eigenvalue sequences and Berezin transforms of radial Bergman measures were recently obtained by Maximenko and Pacheco \cite{MaximenkoPacheco}.

There are also earlier counterexamples of a different type. Zhao and Zheng \cite{ZhaoZheng} constructed Bergman-space symbols for which $\widetilde f\ge 0$ while $T_f\ngeq 0$, and Zhao \cite{ZhaoFock} obtained the analogous phenomenon on the Fock space. As Perälä and Virtanen observed, however, those examples do not address essential positivity. A related comparison appears in \cite{ChenLengZhao}, where a lower bound on the Berezin transform is shown not to imply invertibility for Toeplitz operators induced by positive measures on the Bergman and Fock spaces.

The paper is organized as follows. We begin with the one-dimensional Fock space, where the basic mechanism is clearest, then turn to the one-dimensional Bergman space, and finally pass to higher dimensions. Throughout, the argument stays entirely within the radial setting and avoids indirect operator-theoretic obstructions. Instead, the proofs reduce to explicit asymptotics for diagonal eigenvalues and Berezin transforms.

\section{Fock space}
\subsection{Radial Toeplitz operators on the Fock space}

Let $dA$ denote Lebesgue area measure on $\mathbb C$. We work on the standard Fock space
\[
F^2(\mathbb C)
= \left\{
h\text{ entire}:
\|h\|_{F^2}^2
:= \frac1\pi\int_{\mathbb C}|h(z)|^2e^{-|z|^2}\,dA(z)
<\infty
\right\}.
\]
For $f\in L^\infty(\mathbb C)$, let $T_f$ be the Toeplitz operator
\[
T_f h=P(fh),
\]
where $P$ is the orthogonal projection from
\[
 L^2 \left(\mathbb C,\pi^{-1}e^{-|z|^2}\,dA(z)\right)
\]
onto $F^2(\mathbb C)$.

The standard orthonormal basis is
\[
e_n(z)=\frac{z^n}{\sqrt{n!}}, \qquad n\ge 0.
\]
The normalized reproducing kernel at $z\in\mathbb C$ is
\[
k_z(w)=e^{w\overline z-|z|^2/2},
\]
and the Berezin transform of $f$, which is denoted by a tilde, is
\[
\widetilde f(z)
=\langle T_f k_z,k_z\rangle
=\frac1\pi\int_{\mathbb C} f(w)e^{-|w-z|^2}\,dA(w).
\]
If $f$ is real-valued, then $T_f$ is self-adjoint. As in \cite{PeralaVirtanen}, we say that $T_f$ is essentially positive if
\[
\sigma_{\mathrm{ess}}(T_f)\subset [0,\infty).
\]

\begin{lemma}\label{lem:radial-diagonal}
Let $f(z)=\varphi(|z|)$ be a bounded radial symbol. Then $T_f$ is diagonal with respect to the basis $(e_n)_{n\ge 0}$, and
\[
T_f e_n=\lambda_n e_n,
\qquad
\lambda_n
=\frac1{n!}\int_0^\infty \varphi(\sqrt t)e^{-t}t^n\,dt.
\]
\end{lemma}

\begin{proof}
For $m,n\ge 0$,
\[
\langle T_f e_m,e_n\rangle
=\frac{1}{\pi\sqrt{m!n!}}
\int_{\mathbb C} f(w)w^m\overline w^n e^{-|w|^2}\,dA(w).
\]
In polar coordinates $w=re^{i\theta}$,
\[
\langle T_f e_m,e_n\rangle
=\frac{1}{\pi\sqrt{m!n!}}
\int_0^\infty\int_0^{2\pi}
\varphi(r)r^{m+n}e^{i(m-n)\theta}e^{-r^2}r\,d\theta\,dr.
\]
The angular integral is $0$ when $m\ne n$, so $T_f$ is diagonal. For $m=n$,
\[
\lambda_n
=\frac{2}{n!}\int_0^\infty \varphi(r)e^{-r^2}r^{2n+1}\,dr.
\]
The change of variables $t=r^2$ gives
\[
\lambda_n
=\frac1{n!}\int_0^\infty \varphi(\sqrt t)e^{-t}t^n\,dt.
\]
\end{proof}

\subsection{A radial counterexample on the Fock space}

Consider the bounded real-valued radial symbol
\[
f(z)=\frac12+\cos(2|z|), \qquad z\in\mathbb C.
\]

\begin{theorem}\label{thm:fock-counterexample}
For the symbol $f(z)=\frac12+\cos(2|z|)$,
\[
\liminf_{|z|\to\infty}\widetilde f(z)=\frac12-e^{-1}>0,
\]
but $T_f$ is not essentially positive.
\end{theorem}

The proof is divided into two asymptotic formulas.

\begin{proposition}\label{prop:eigenvalue-asymptotic}
Let $(\lambda_n)_{n\ge 0}$ be the eigenvalue sequence from Lemma \ref{lem:radial-diagonal}. Then
\[
\lambda_n
=\frac12+e^{-1/2}\cos(2\sqrt{n+1})+o(1)
\qquad (\text{as } n\to\infty).
\]
Thus
\[
\liminf_{n\to\infty}\lambda_n
=\frac12-e^{-1/2}<0.
\]
\end{proposition}

\begin{proof}
Lemma \ref{lem:radial-diagonal} gives
\[
\lambda_n
=\frac12+\Re I_n,
\qquad
I_n
=\frac1{n!}\int_0^\infty e^{-t}t^n e^{2i\sqrt t}\,dt.
\]
Set
\[
a=\sqrt{n+1}.
\]
Perform a second-order Taylor expansion of 2$\sqrt t$ around $t = a^2$ with exact remainder 
\[
2\sqrt t
=
2a+\frac{t-a^2}{a}-\frac{(\sqrt t-a)^2}{a}
\]
or, alternatively, view this as an algebraic identity. Hence
\[
I_n=M_n+R_n,
\]
where
\[
M_n
=e^{2ia}\frac1{n!}\int_0^\infty e^{-t}t^n e^{i(t-a^2)/a}\,dt
\]
and
\[
R_n
=e^{2ia}\frac1{n!}\int_0^\infty e^{-t}t^n e^{i(t-a^2)/a}
\left(e^{-i(\sqrt t-a)^2/a}-1\right)\,dt.
\]

We first estimate $R_n$. The bound $|e^{iu}-1|\le |u|$ gives
\[
|R_n|
\le\frac1{a\,n!}\int_0^\infty e^{-t}t^n(\sqrt t-a)^2\,dt.
\]
Since
\[
(\sqrt t-a)^2
=\frac{(t-a^2)^2}{(\sqrt t+a)^2}
\le\frac{(t-a^2)^2}{a^2},
\]
it follows that
\[
|R_n|
\le\frac1{a^3n!}\int_0^\infty e^{-t}t^n(t-a^2)^2\,dt.
\]
The last integral can be computed directly:
\[
\frac1{n!}\int_0^\infty e^{-t}t^n(t-a^2)^2\,dt
=
\frac{(n+2)!}{n!}
-2a^2\frac{(n+1)!}{n!}
+a^4.
\]
Because $a^2=n+1$, this equals $(n+2)(n+1)-2(n+1)^2+(n+1)^2=n+1=a^2$.
Thus
\[
|R_n|\le \frac1a\to 0.
\]

For the main term,
\[
M_n
=e^{2ia}e^{-ia}\frac1{n!}\int_0^\infty e^{-(1-i/a)t}t^n\,dt
=e^{ia}(1-i/a)^{-(n+1)} = e^{ia}(1-i/a)^{-a^2}.
\]
Using the principal branch of the logarithm,
\[
M_n=\exp \big(ia-a^2\log(1-i/a)\big).
\]
The expansion
\[
\log(1-z)=-z-\frac{z^2}{2}+O(|z|^3)
\qquad (z\to 0)
\]
gives
\[
\log(1-i/a)
=-\frac{i}{a}+\frac{1}{2a^2}+O(a^{-3}),
\]
and therefore
\[
ia-a^2\log(1-i/a)
=2ia-\frac12+O(a^{-1}).
\]
Hence
\[
M_n=e^{-1/2}e^{2ia}+o(1).
\]
Since $R_n=o(1)$, we obtain
\[
I_n=e^{-1/2}e^{2i\sqrt{n+1}}+o(1),
\]
so
\[
\lambda_n
=\frac12+e^{-1/2}\cos(2\sqrt{n+1})+o(1).
\]

This implies
\[
\liminf_{n\to\infty}\lambda_n\ge \frac12-e^{-1/2}.
\]
To get the reverse inequality, define
\[
x_k=\left(\frac{(2k+1)\pi}{2}\right)^2-1,
\qquad
n_k=\lfloor x_k\rfloor.
\]
Then $n_k\to\infty$, and the mean value theorem for the function $f(x)=2\sqrt{x+1}$ yields
\[
\big|2\sqrt{n_k+1}-(2k+1)\pi\big|
\le
\sup_{x\in[n_k,x_k]}\frac1{\sqrt{x+1}}
\to 0.
\]
Hence $\cos(2\sqrt{n_k+1})\to -1$ and along this subsequence,
\[
\lambda_{n_k}\to \frac12-e^{-1/2}.
\]
Therefore
\[
\liminf_{n\to\infty}\lambda_n=\frac12-e^{-1/2}<0.
\]
\end{proof}

\begin{proposition}\label{prop:berezin-asymptotic}
The Berezin transform of $f$ satisfies
\[
\widetilde f(z)
=\frac12+e^{-1}\cos(2|z|)+o(1)
\qquad (\text{as } |z|\to\infty).
\]
Thus
\[
\liminf_{|z|\to\infty}\widetilde f(z)
=\frac12-e^{-1}>0.
\]
\end{proposition}

\begin{proof}
Since $f$ is radial, $\widetilde f$ is radial as well. Indeed, for $\theta\in\mathbb R$,
\[
\widetilde f(e^{i\theta}z)
=\frac1\pi\int_{\mathbb C} f(w)e^{-|w-e^{i\theta}z|^2}\,dA(w)
=\frac1\pi\int_{\mathbb C} f(e^{i\theta}u)e^{-|u-z|^2}\,dA(u)
=\widetilde f(z).
\]
It is therefore enough to consider $z=s>0$ real.

Write $w=s+x+iy$. Then $dA(w)=dx\,dy$ and $|w-s|^2=x^2+y^2$. Hence
\[
\widetilde f(s)
=\frac12+\Re \big(e^{2is}J_s\big),
\]
where
\[
J_s
=\frac1\pi\int_{\mathbb R^2}
e^{-x^2-y^2}
e^{2i(\sqrt{(s+x)^2+y^2}-s)}
\,dx\,dy.
\]

Fix $x,y\in\mathbb R$. Then
\[
\sqrt{(s+x)^2+y^2}-s
=\frac{2sx+x^2+y^2}{\sqrt{(s+x)^2+y^2}+s}
\longrightarrow x
\qquad (\text{as } s\to\infty).
\]
The integrand is bounded in absolute value by $\frac1\pi e^{-x^2-y^2},$ which is integrable on $\mathbb R^2$. By dominated convergence,
\[
J_s\to
\frac1\pi\int_{\mathbb R^2}e^{-x^2-y^2}e^{2ix}\,dx\,dy.
\]
The $y$-integral is $\sqrt\pi$. For the $x$-integral, define
\[
F(t)=\int_{\mathbb R} e^{-x^2}\cos(tx)\,dx.
\]
Differentiating under the integral sign and integrating by parts,
\[
F'(t)
=-\int_{\mathbb R} x e^{-x^2}\sin(tx)\,dx
=-\frac t2\int_{\mathbb R} e^{-x^2}\cos(tx)\,dx
=-\frac t2 F(t).
\]
Since $F(0)=\sqrt\pi$, 
\[
F(t)=\sqrt\pi e^{-t^2/4}.
\]
The sine integral vanishes by oddness, so $\int_{\mathbb R} e^{-x^2}e^{itx}\,dx=F(t)$, which evaluates to $\sqrt{\pi}/e$ at $t=2$. Thus
\[
\frac1\pi\int_{\mathbb R^2}e^{-x^2-y^2}e^{2ix}\,dx\,dy=e^{-1}.
\]
Thus $J_s\to e^{-1}$, and
\[
\widetilde f(s)
=\frac12+e^{-1}\cos(2s)+o(1).
\]

This yields
\[
\liminf_{s\to\infty}\widetilde f(s)\ge \frac12-e^{-1}.
\]
For the reverse inequality, take
\[
s_k=\frac{(2k+1)\pi}{2}.
\]
Then $\cos(2s_k)=-1$, so
\[
\widetilde f(s_k)\to \frac12-e^{-1}.
\]
Hence
\[
\liminf_{|z|\to\infty}\widetilde f(z)
= \liminf_{s\to\infty}\widetilde f(s)
= \frac12-e^{-1}>0.
\]
\end{proof}

\begin{proof}[Proof of Theorem \ref{thm:fock-counterexample}]
Proposition \ref{prop:berezin-asymptotic} gives
\[
\liminf_{|z|\to\infty}\widetilde f(z)=\frac12-e^{-1}>0.
\]

By Proposition \ref{prop:eigenvalue-asymptotic}, there exists a sequence $n_k\to\infty$ such that
\[
\lambda_{n_k}\to \frac12-e^{-1/2}<0.
\]
Since $T_f e_n=\lambda_n e_n$, we have
\[
(T_f-(1/2-e^{-1/2})I)e_{n_k}\to 0.
\]
The vectors $e_{n_k}$ are orthonormal, hence they converge weakly to $0$. Weyl's criterion now gives
\[
\frac12-e^{-1/2}\in \sigma_{\mathrm{ess}}(T_f).
\]
This point of the essential spectrum is negative, so $T_f$ is not essentially positive.
\end{proof}

\begin{remark}\label{rem:fock-param}
The proofs of Propositions~\ref{prop:eigenvalue-asymptotic} and~\ref{prop:berezin-asymptotic} go through with $\beta$ in place of $2$ and $c$ in place of $\tfrac{1}{2}$, giving the following. For\[
f_{\beta,c}(z) = c + \cos(\beta|z|), \qquad \beta > 0,
\]
one obtains
\[
\lambda_n
=c+e^{-\beta^2/8}\cos(\beta\sqrt{n+1})+o(1)
\qquad (\text{as } n\to\infty),
\]
and
\[
\widetilde f_{\beta,c}(z)
=
c+e^{-\beta^2/4}\cos(\beta|z|)+o(1)
\qquad (\text{as } |z|\to\infty).
\]
Any choice of $c$ satisfying
\[
e^{-\beta^2/4}<c<e^{-\beta^2/8}
\]
gives a radial counterexample. Thus the set of counterexamples is open in the $(\beta,c)$ parameter space.
\end{remark}

\section{Bergman space}
\subsection{Radial Toeplitz operators on the Bergman space}

Let $dA$ denote Lebesgue measure on the disk $\mathbb D$. We work on the standard Bergman space
\[
A^2(\mathbb D)
=
\left\{
h\in \operatorname{Hol}(\mathbb D):
\|h\|_{A^2}^2
:=
\frac1\pi\int_{\mathbb D}|h(z)|^2\,dA(z)
<\infty
\right\}.
\]
For $f\in L^\infty(\mathbb D)$, let $T_f$ be the Toeplitz operator
\[
T_f h=P(fh),
\]
where $P$ is the orthogonal projection from
\[
L^2 \left(\mathbb D,\pi^{-1}dA\right)
\]
onto $A^2(\mathbb D)$.

The standard orthonormal basis is
\[
e_n(z)=\sqrt{n+1}\,z^n,
\qquad n\ge 0.
\]
The reproducing kernel is
\[
K_z(w)=\frac{1}{(1-w\overline z)^2},
\]
and the normalized reproducing kernel at $z\in\mathbb D$ is
\[
k_z(w)=\frac{1-|z|^2}{(1-w\overline z)^2}.
\]
The Berezin transform of $f$ is
\[
\widetilde f(z)
=\langle T_f k_z,k_z\rangle
=
\frac1\pi\int_{\mathbb D}
f(w)\frac{(1-|z|^2)^2}{|1-w\overline z|^4}\,dA(w).
\]
If $f$ is real-valued, then $T_f$ is self-adjoint. As in the Fock space case, we say that $T_f$ is essentially positive if
\[
\sigma_{\mathrm{ess}}(T_f)\subset [0,\infty).
\]

\begin{lemma}\label{lem:bergman-radial}
Let $f(z)=\varphi(|z|^2)$ be a bounded radial symbol. Then $T_f$ is diagonal with respect to the basis $(e_n)_{n\ge 0}$, and
\[
T_f e_n=\lambda_n e_n,
\qquad
\lambda_n
=(n+1)\int_0^1 \varphi(t)t^n\,dt.
\]
Moreover, $\widetilde f$ is radial, and for $0\le a<1$,
\[
\widetilde f(a)
=(1-a^2)^2
\int_0^1
\frac{1+a^2 t}{(1-a^2 t)^3}\,\varphi(t)\,dt.
\]
\end{lemma}

\begin{proof}
For $m,n\ge 0$,
\[
\langle T_f e_m,e_n\rangle
=
\frac{\sqrt{(m+1)(n+1)}}{\pi}
\int_{\mathbb D}
f(w)w^m\overline w^n\,dA(w).
\]
In polar coordinates $w=re^{i\theta}$,
\[
\langle T_f e_m,e_n\rangle
=
\frac{\sqrt{(m+1)(n+1)}}{\pi}
\int_0^1\int_0^{2\pi}
\varphi(r^2)r^{m+n}e^{i(m-n)\theta}r\,d\theta\,dr.
\]
The angular integral is $0$ when $m\ne n$, so $T_f$ is diagonal. For $m=n$,
\[
\lambda_n=2(n+1)\int_0^1 \varphi(r^2)r^{2n+1}\,dr=(n+1)\int_0^1 \varphi(t)t^n\,dt
\]with the latter by the change of variables $t=r^2$.

Since $f$ is radial, $\widetilde f$ is radial as well. It is therefore enough to take $z=a\in[0,1)$ real. Then
\[
\widetilde f(a)
=
\frac{(1-a^2)^2}{\pi}
\int_0^1\int_0^{2\pi}
\frac{\varphi(r^2)}{|1-ar e^{-i\theta}|^4}r\,d\theta\,dr.
\]
For $0\le \rho<1$,
\[
\frac1{|1-\rho e^{i\theta}|^4}
=
\frac1{(1-\rho e^{i\theta})^2(1-\rho e^{-i\theta})^2}
=
\sum_{m,n\ge 0}(m+1)(n+1)\rho^{m+n}e^{i(m-n)\theta}.
\]
Averaging in $\theta$ yields
\[
\frac1{2\pi}\int_0^{2\pi}\frac{d\theta}{|1-\rho e^{i\theta}|^4}
=\sum_{n\ge 0}(n+1)^2\rho^{2n}=\frac{1+\rho^2}{(1-\rho^2)^3}.
\]
Applying this with $\rho=ar$, we obtain
\[
\widetilde f(a)=2(1-a^2)^2
\int_0^1
\frac{1+a^2r^2}{(1-a^2r^2)^3}\,\varphi(r^2)\,r\,dr.
\]
Now set $t=r^2$. Then
\[
\widetilde f(a)=(1-a^2)^2
\int_0^1
\frac{1+a^2 t}{(1-a^2 t)^3}\,\varphi(t)\,dt.
\]
\end{proof}

\subsection{A radial counterexample on the Bergman space}

On the Bergman space, the analogue of the Fock example is to oscillate at the boundary scale $1-|z|^2$, not at the Euclidean scale $|z|$. Replacing Euclidean oscillation with oscillation at the hyperbolic boundary scale $1-|z|^2$ follows the standard dictionary between the two spaces.

Consider the bounded real-valued radial symbol
\[
f(z)=\frac12+\cos \left(\log \frac1{1-|z|^2}\right),
\qquad z\in\mathbb D.
\]
Equivalently,
\[
f(z)=\frac12+\Re (1-|z|^2)^{-i},
\]
where for $u>0$ we write
\[
u^{-i}=e^{-i\log u}.
\]
All occurrences of $u^{-i}$ and $(1-t)^{-i}$ below are understood for $u>0$ and $0<t<1$; the endpoint values are irrelevant for the integrals.

\begin{theorem}\label{thm:bergman-counterexample}
For the symbol
\[
f(z)=\frac12+\cos \left(\log \frac1{1-|z|^2}\right),
\]
one has
\[
\liminf_{|z|\to 1}\widetilde f(z)=\frac12-\sqrt2\,\frac{\pi}{\sinh\pi} >0,
\]
but $T_f$ is not essentially positive.
\end{theorem}

The proof is divided into two asymptotic formulas.

\begin{proposition}\label{prop:bergman-eigenvalue-asymptotic}
Let $(\lambda_n)_{n\ge 0}$ be the eigenvalue sequence from Lemma \ref{lem:bergman-radial}, and set
\[
\alpha=\int_0^\infty e^{-u}u^{-i}\,du.
\]
Then
\[
\lambda_n=\frac12+\Re \big(\alpha (n+1)^i\big)+o(1)
\qquad (\text{as } n\to\infty).
\]
In particular,
\[
\liminf_{n\to\infty}\lambda_n=\frac12-|\alpha|.
\]
\end{proposition}

\begin{proof}
Lemma \ref{lem:bergman-radial} gives
\[
\lambda_n=\frac12+\Re I_n,
\qquad
I_n=(n+1)\int_0^1 (1-t)^{-i}t^n\,dt.
\]
Setting $u=(n+1)(1-t)$, then
\[
I_n=(n+1)^i
\int_0^\infty
u^{-i}\left(1-\frac{u}{n+1}\right)^n
\mathbf 1_{[0,n+1]}(u)\,du.
\]
For each fixed $u>0$,
\[
\left(1-\frac{u}{n+1}\right)^n\mathbf 1_{[0,n+1]}(u)\to e^{-u}
\qquad (\text{as } n\to\infty).
\]
Also, for $n\ge 1$,
\[
0\le
\left(1-\frac{u}{n+1}\right)^n\mathbf 1_{[0,n+1]}(u)
\le
e^{-nu/(n+1)}
\le
e^{-u/2}.
\]
Since $|u^{-i}|=1$, dominated convergence gives
\[
I_n=\alpha (n+1)^i+o(1).
\]
Therefore
\[
\lambda_n=\frac12+\Re \big(\alpha (n+1)^i\big)+o(1).
\]

Since
\[
\Re \big(\alpha (n+1)^i\big)\ge -|\alpha|,
\]
we obtain
\[
\liminf_{n\to\infty}\lambda_n\ge \frac12-|\alpha|.
\]
To get the reverse inequality, write $\alpha=|\alpha|e^{i\theta}$ and define
\[
x_k=e^{(2k+1)\pi-\theta}-1,
\qquad
n_k=\lfloor x_k\rfloor.
\]
Then $n_k\to\infty$, and the mean value theorem gives
\[
\big|\log(n_k+1)-\log(x_k+1)\big|\le
\frac{x_k-n_k}{n_k+1}
\le\frac1{n_k+1}
\to 0.
\]
Since $\log(x_k+1)+\theta=(2k+1)\pi$, it follows that
\[
\cos\big(\log(n_k+1)+\theta\big)\to -1.
\]
Along this subsequence,
\[
\lambda_{n_k}\to \frac12-|\alpha|.
\]
Hence
\[
\liminf_{n\to\infty}\lambda_n
=\frac12-|\alpha|.
\]
\end{proof}

\begin{proposition}\label{prop:bergman-berezin-asymptotic}
Set
\[
\beta
=2\int_0^\infty \frac{u^{-i}}{(1+u)^3}\,du.
\]
Then the Berezin transform of $f$ satisfies
\[
\widetilde f(a)
= \frac12+\Re \big(\beta (1-a^2)^{-i}\big)+o(1)
\qquad (\text{as } a\to 1^-).
\]
In particular,
\[
\liminf_{a\to 1^-}\widetilde f(a)
=\frac12-|\beta|.
\]
\end{proposition}

\begin{proof}
Lemma \ref{lem:bergman-radial} gives
\[
\widetilde f(a)=\frac12+\Re J_a,
\]
where
\[
J_a=(1-a^2)^2
\int_0^1
\frac{1+a^2 t}{(1-a^2 t)^3}(1-t)^{-i}\,dt.
\]
Set
\[
\delta=1-a^2.
\]
Then $a^2=1-\delta$, and with the change of variables $t=1-\delta u$,
\[
J_a=\delta^{-i}
\int_0^\infty
u^{-i}
\frac{2-\delta-\delta u+\delta^2u}{(1+u-\delta u)^3}
\mathbf 1_{[0,1/\delta]}(u)\,du.
\]
For each fixed $u > 0$,
\[
u^{-i}
\frac{2-\delta-\delta u+\delta^2u}{(1+u-\delta u)^3}
\mathbf 1_{[0,1/\delta]}(u)
\longrightarrow
\frac{2u^{-i}}{(1+u)^3}
\qquad (\text{as } \delta\to 0^+).
\]
If $\delta\le \frac12$, then $1+u-\delta u\ge 1+\frac u2$, while
\[
|2-\delta-\delta u+\delta^2u|\le 3+u.
\]
Hence
\[
\left|
u^{-i}
\frac{2-\delta-\delta u+\delta^2u}{(1+u-\delta u)^3}
\mathbf 1_{[0,1/\delta]}(u)
\right|
\le\frac{C}{(1+u)^2}
\]
for some absolute constant $C$, and the right-hand side is integrable on $[0,\infty)$. Dominated convergence therefore gives
\[
J_a
=\beta\,\delta^{-i}+o(1)
=\beta(1-a^2)^{-i}+o(1).
\]
Thus
\[
\widetilde f(a)
=\frac12+\Re \big(\beta(1-a^2)^{-i}\big)+o(1).
\]

Since
\[
\Re \big(\beta(1-a^2)^{-i}\big)\ge -|\beta|,
\]
we obtain
\[
\liminf_{a\to 1^-}\widetilde f(a)\ge \frac12-|\beta|.
\]
Write $\beta=|\beta|e^{i\phi}$. For large $k$, define
\[
\delta_k=e^{-((2k+1)\pi-\phi)},
\qquad
a_k=\sqrt{1-\delta_k}.
\]
Then $a_k\to 1^-$ and
\[
\log \frac1{1-a_k^2}+\phi
=(2k+1)\pi.
\]
Hence
\[
\cos \left(\log \frac1{1-a_k^2}+\phi\right)=-1,
\]
so
\[
\widetilde f(a_k)\to \frac12-|\beta|.
\]
Therefore
\[
\liminf_{a\to 1^-}\widetilde f(a)
=\frac12-|\beta|.
\]
\end{proof}

\begin{lemma}\label{lem:bergman-constants}
With $\alpha$ and $\beta$ as above,
\[
|\alpha|=
\sqrt{\frac{\pi}{\sinh\pi}},
\qquad
|\beta|
=\sqrt2\,\frac{\pi}{\sinh\pi}.
\]
In particular,
\[
|\alpha|> \frac12
\qquad\text{and}\qquad
|\beta|<\frac12.
\]
\end{lemma}

\begin{proof}
First,
\[
\frac{2}{(1+u)^3}
=\int_0^\infty s^2 e^{-s}e^{-su}\,ds.
\]
Using Fubini,
\[
\beta=
\int_0^\infty s^2e^{-s}
\left(\int_0^\infty u^{-i}e^{-su}\,du\right)\,ds.
\]
In the inner integral, the change of variables $v=su$ gives
\[
\int_0^\infty u^{-i}e^{-su}\,du
=
s^{i-1}\int_0^\infty e^{-v}v^{-i}\,dv
=
s^{i-1}\alpha.
\]
Hence
\[
\beta=\alpha\int_0^\infty s^{1+i}e^{-s}\,ds.
\]
Integrating by parts,
\[
\int_0^\infty s^{1+i}e^{-s}\,ds
= (1+i)\int_0^\infty s^i e^{-s}\,ds
= (1+i)\overline\alpha.
\]
Therefore
\[
\beta=(1+i)|\alpha|^2,
\]
and so
\[
|\beta|=\sqrt2\,|\alpha|^2.
\]

Next, let
\[
\Gamma(w)=\int_0^\infty e^{-s}s^{w-1}\,ds
\qquad (\Re w>0).
\]
Then $\alpha=\Gamma(1-i)$. By analytic continuation, the reflection formula
\[
\Gamma(z)\Gamma(1-z)=\frac{\pi}{\sin(\pi z)}
\]
at $z=i$, we obtain
\[
\Gamma(i)\Gamma(1-i)
=\frac{\pi}{\sin(\pi i)}=-i\frac{\pi}{\sinh\pi}.
\]
Since $\Gamma(1+i)=i\Gamma(i),$ it follows that
\[
|\alpha|^2
=\Gamma(1-i)\Gamma(1+i)
= i \Gamma(1-i)\Gamma(i)
= \frac{\pi}{\sinh\pi}.
\]
Thus
\[
|\alpha|=\sqrt{\frac{\pi}{\sinh\pi}},
\qquad
|\beta|=\sqrt2\,\frac{\pi}{\sinh\pi}.
\]
Numerically,
\[
|\alpha|\approx 0.52,
\qquad
|\beta|\approx 0.38.
\]
In particular,
\[
|\alpha|>\frac12
\qquad\text{and}\qquad
|\beta|<\frac12.
\]
\end{proof}

\begin{proof}[Proof of Theorem \ref{thm:bergman-counterexample}]
Proposition \ref{prop:bergman-berezin-asymptotic} and Lemma \ref{lem:bergman-constants} give
\[
\liminf_{|z|\to 1}\widetilde f(z)
=\frac12-|\beta|
=\frac12-\sqrt2\,\frac{\pi}{\sinh\pi}>0.
\]

By Proposition \ref{prop:bergman-eigenvalue-asymptotic} and Lemma \ref{lem:bergman-constants}, there exists a sequence $n_k\to\infty$ such that
\[
\lambda_{n_k}\to \frac12-|\alpha| = \frac12-\sqrt{\frac{\pi}{\sinh\pi}}<0.
\]
Since
\[
T_f e_n=\lambda_n e_n,
\]
we have
\[
\left(T_f-\left(\frac12-|\alpha|\right)I\right)e_{n_k}\to 0.
\]
The vectors $e_{n_k}$ are orthonormal, hence they converge weakly to $0$. By Weyl's criterion,
\[
\frac12-|\alpha|\in \sigma_{\mathrm{ess}}(T_f).
\]
This point of the essential spectrum is negative, so $T_f$ is not essentially positive.
\end{proof}

\begin{remark}
The whole argument is driven by the boundary scales
\[
1-t\asymp \frac1{n+1}
\qquad\text{and}\qquad
1-t\asymp 1-a^2.
\]
The two asymptotics come directly from these rescalings and dominated convergence. The only special-function input is the final evaluation of $|\alpha|$ in Lemma \ref{lem:bergman-constants}. In the Fock space, the Fock analogue is $t \asymp n+1$ versus $t \asymp |z|^2$. 
\end{remark}

\begin{remark}\label{rem:bergman-param}
The Bergman counterexample admits the same kind of parametric generalization as Remark~\ref{rem:fock-param}.
For $\omega > 0$ and $c \in \mathbb{R}$, set
\[
  f_{\omega,c}(z)
  = c + \cos \Bigl(\omega\log\frac{1}{1-|z|^2}\Bigr)
  = c + \Re(1-|z|^2)^{-i\omega}.
\]
The proofs of Propositions~\ref{prop:bergman-eigenvalue-asymptotic}
and~\ref{prop:bergman-berezin-asymptotic} go through with $\omega$ in place of $1$.
Writing
\[
  \alpha_\omega = \Gamma(1-i\omega),
\]
one obtains
\[
  \lambda_n
  = c + \Re\bigl(\alpha_\omega (n+1)^{i\omega}\bigr) + o(1)
  = c + |\alpha_\omega|
    \cos\bigl(\omega\log(n+1)+\arg\alpha_\omega\bigr) + o(1)
  \qquad (\text{as } n\to\infty),
\]
and
\[
  \widetilde f_{\omega,c}(a)
  = c + \Re\bigl(\beta_\omega (1-a^2)^{-i\omega}\bigr) + o(1)
  \qquad (\text{as } a\to1^-),
\]
where
\[
  \beta_\omega
  = 2\int_0^\infty \frac{u^{-i\omega}}{(1+u)^3}\,du
  = \Gamma(1-i\omega)\Gamma(2+i\omega)
  = (1+i\omega)|\alpha_\omega|^2.
\]
Equivalently,
\[
  \widetilde f_{\omega,c}(a)
  = c + \sqrt{1+\omega^2}\,|\alpha_\omega|^2
    \cos \Bigl(\omega\log\frac{1}{1-a^2}+\arctan\omega\Bigr) + o(1).
\]
Any choice of $c$ satisfying
\[
  \sqrt{1+\omega^2}\,|\alpha_\omega|^2 < c < |\alpha_\omega|
\]
gives a radial counterexample on the Bergman space.
Moreover,
\[
  |\alpha_\omega|^2 = \frac{\pi\omega}{\sinh(\pi\omega)},
\]
so this interval is nonempty for every $\omega>0$.
Hence the set of such $(\omega,c)$ is open, and for each $\omega>0$ it is nonempty.
\end{remark}

\subsection{Higher-dimensional variants}

In this subsection let $d\ge 1$, let $\mathbb B_d=\{z\in\C^d:|z|<1\}$, and let $dV$ denote Lebesgue volume measure on $\C^d$. We use the standard normalizations
\[
\|h\|_{F^2(\C^d)}^2=\frac1{\pi^d}\int_{\C^d}|h(z)|^2e^{-|z|^2}\,dV(z),
\qquad
\|h\|_{A^2(\mathbb B_d)}^2=\frac{d!}{\pi^d}\int_{\mathbb B_d}|h(z)|^2\,dV(z).
\]

\begin{proposition}\label{prop:fock-higher-dim}
On $F^2(\C^d)$, the symbol
\[
f(z)=\frac12+\cos(2|z|)
\]
satisfies
\[
\liminf_{|z|\to\infty}\widetilde f(z)=\frac12-e^{-1}>0,
\]
but $T_f$ is not essentially positive.
\end{proposition}

\begin{proof}
Write
\[
e_\mu(z)=\frac{z^\mu}{\sqrt{\mu!}},
\qquad
\mu\in\mathbb N^d.
\]
If $\mu\ne \nu$, choose $j$ with $\mu_j\ne \nu_j$. Rotating the $j$th coordinate shows that
\[
\langle T_f z^\mu,z^\nu\rangle=0.
\]
Thus $T_f$ is diagonal in the monomial basis. If $|\mu|=m$, then
\[
T_f e_\mu=\lambda_m e_\mu
\]
with
\[
\lambda_m=
\frac{\int_0^\infty (\frac12+\cos(2r))e^{-r^2}r^{2m+2d-1}\,dr}
{\int_0^\infty e^{-r^2}r^{2m+2d-1}\,dr}
=
\frac1{(m+d-1)!}\int_0^\infty \left(\frac12+\cos(2\sqrt t)\right)e^{-t}t^{m+d-1}\,dt.
\]
This quotient is independent of the choice of $\mu$ with $|\mu|=m$, since in polar coordinates $z=r\zeta$ the same angular factor $\int_{S^{2d-1}}|\zeta^\mu|^2\,d\sigma(\zeta)$ appears in both numerator and denominator and cancels.
The proof of Proposition \ref{prop:eigenvalue-asymptotic} applies verbatim with $n$ replaced by $m+d-1$. Hence
\[
\lambda_m=\frac12+e^{-1/2}\cos(2\sqrt{m+d})+o(1)
\]
and therefore
\[
\liminf_{m\to\infty}\lambda_m=\frac12-e^{-1/2}<0.
\]

The Berezin transform is radial by unitary invariance, so it is enough to take $z=se_1$, where $e_1=(1,0,\dots,0)$. Writing $w=se_1+u$ gives
\[
\widetilde f(se_1)=\frac12+\Re(e^{2is}J_s),
\]
where
\[
J_s=\frac1{\pi^d}\int_{\C^d} e^{-|u|^2}e^{2i(|se_1+u|-s)}\,dV(u).
\]
For each fixed $u\in\C^d$,
\[
|se_1+u|-s=\frac{2s\Re u_1+|u|^2}{|se_1+u|+s}\to \Re u_1
\]
as $s\to\infty$. Since the integrand is bounded by $\pi^{-d}e^{-|u|^2}$, dominated convergence gives
\[
J_s\to \frac1{\pi^d}\int_{\C^d} e^{-|u|^2}e^{2i\Re u_1}\,dV(u)=e^{-1}.
\]
Thus
\[
\widetilde f(z)=\frac12+e^{-1}\cos(2|z|)+o(1)
\]
as $|z|\to\infty$, and so
\[
\liminf_{|z|\to\infty}\widetilde f(z)=\frac12-e^{-1}>0.
\]

Choose integers $m_k\to\infty$ with $\lambda_{m_k}\to \frac12-e^{-1/2}$. Then
\[
\left(T_f-\left(\frac12-e^{-1/2}\right)I\right)e_{(m_k,0,\dots,0)}\to 0.
\]
The vectors $e_{(m_k,0,\dots,0)}$ are orthonormal, hence weakly null. Weyl's criterion gives
\[
\frac12-e^{-1/2}\in \sigma_{\mathrm{ess}}(T_f),
\]
so $T_f$ is not essentially positive.
\end{proof}

\begin{proposition}\label{prop:bergman-higher-dim}
Let
\[
\alpha=\Gamma(1-i), \quad
\beta_d=\frac{\Gamma(1-i)\Gamma(d+1+i)}{\Gamma(d+1)}, \quad
c_d=\frac{|\alpha|+|\beta_d|}{2}.
\]
On $A^2(\mathbb B_d)$, the symbol
\[
f_d(z)=c_d+\cos\left(\log \frac1{1-|z|^2}\right)
\]
satisfies
\[
\liminf_{|z|\to 1^-}\widetilde f_d(z)=c_d-|\beta_d|>0,
\]
but $T_{f_d}$ is not essentially positive.
\end{proposition}

\begin{proof}
Write
\[
e_\mu(z)=\left(\frac{(d+|\mu|)!}{\mu!\,d!}\right)^{1/2}z^\mu,
\qquad
\mu\in\mathbb N^d.
\]
If $\mu\ne \nu$, the same rotation argument shows that
\[
\langle T_{f_d} z^\mu,z^\nu\rangle=0.
\]
Thus $T_{f_d}$ is diagonal in the monomial basis. If $|\mu|=m$, then
\[
T_{f_d}e_\mu=\lambda_m e_\mu
\]
with
\[
\lambda_m=
\frac{\int_0^1 \left(c_d+\cos(\log \frac1{1-r^2})\right)r^{2m+2d-1}\,dr}
{\int_0^1 r^{2m+2d-1}\,dr}
=
(m+d)\int_0^1 \left(c_d+\cos\left(\log \frac1{1-t}\right)\right)t^{m+d-1}\,dt.
\]
Again, this quotient is independent of the choice of $\mu$ with $|\mu|=m$, because in polar coordinates $z=r\zeta$ the common angular factor $\int_{S^{2d-1}}|\zeta^\mu|^2\,d\sigma(\zeta)$ cancels.
Hence
\[
\lambda_m=c_d+\Re I_m,
\qquad
I_m=(m+d)\int_0^1 (1-t)^{-i}t^{m+d-1}\,dt.
\]
By the beta integral,
\[
I_m=\Gamma(1-i)\frac{\Gamma(m+d+1)}{\Gamma(m+d+1-i)}=\alpha(m+d)^i+o(1).
\]
Thus
\[
\lambda_m=c_d+\Re(\alpha(m+d)^i)+o(1),
\]
and the same phase-selection argument as in Proposition \ref{prop:bergman-eigenvalue-asymptotic} gives
\[
\liminf_{m\to\infty}\lambda_m=c_d-|\alpha|.
\]

The Berezin transform is radial as well, so it is enough to take $z=ae_1$ with $0\le a<1$. Let $Q_m$ denote the orthogonal projection onto the homogeneous polynomials of degree $m$. Since
\[
k_{ae_1}(w)=\frac{(1-a^2)^{(d+1)/2}}{(1-aw_1)^{d+1}}
=(1-a^2)^{(d+1)/2}\sum_{m\ge 0}\binom{m+d}{m}a^m w_1^m
\]
and
\[
\|w_1^m\|_{A^2(\mathbb B_d)}^2=\frac{m!\,d!}{(m+d)!}=\binom{m+d}{m}^{-1},
\]
we get
\[
\|Q_m k_{ae_1}\|^2=(1-a^2)^{d+1}\binom{m+d}{m}a^{2m}.
\]
Therefore
\[
\widetilde f_d(ae_1)=\sum_{m\ge 0}\lambda_m\|Q_m k_{ae_1}\|^2
=(1-a^2)^{d+1}\sum_{m\ge 0}\lambda_m\binom{m+d}{m}a^{2m}.
\]
Substituting the formula for $\lambda_m$ and using
\[
\sum_{m\ge 0}\binom{m+d}{m}x^m=\frac1{(1-x)^{d+1}},
\qquad
\sum_{m\ge 0}(m+d)\binom{m+d}{m}x^m=\frac{d+x}{(1-x)^{d+2}},
\]
we obtain
\[
\widetilde f_d(ae_1)=
(1-a^2)^{d+1}\int_0^1
t^{d-1}\frac{d+a^2 t}{(1-a^2 t)^{d+2}}
\left(c_d+\cos\left(\log \frac1{1-t}\right)\right)\,dt.
\]
For the oscillatory part,
\[
J_a=
(1-a^2)^{d+1}\int_0^1
t^{d-1}\frac{d+a^2 t}{(1-a^2 t)^{d+2}}(1-t)^{-i}\,dt.
\]
Set $\delta=1-a^2$ and $t=1-\delta u$. Then
\[
J_a=
\delta^{-i}\int_0^{1/\delta}
u^{-i}(1-\delta u)^{d-1}
\frac{d+1-\delta-\delta u+\delta^2u}{(1+u-\delta u)^{d+2}}\,du.
\]
For each fixed $u>0$, the integrand converges to
\[
(d+1)\frac{u^{-i}}{(1+u)^{d+2}}.
\]
If $\delta\le \frac12$, then $1+u-\delta u\ge 1+\frac u2$ and $(1-\delta u)^{d-1}\le 1$ on the interval of integration, so the integrand is bounded by $C_d(1+u)^{-2}$. Dominated convergence gives
\[
J_a=\beta_d(1-a^2)^{-i}+o(1),
\]
where
\[
\beta_d=(d+1)\int_0^\infty \frac{u^{-i}}{(1+u)^{d+2}}\,du
=\frac{\Gamma(1-i)\Gamma(d+1+i)}{\Gamma(d+1)}.
\]
Hence
\[
\widetilde f_d(ae_1)=c_d+\Re(\beta_d(1-a^2)^{-i})+o(1),
\]
and the same phase-selection argument as in Proposition \ref{prop:bergman-berezin-asymptotic} gives
\[
\liminf_{a\to 1^-}\widetilde f_d(ae_1)=c_d-|\beta_d|.
\]

It remains to show $|\beta_d|<|\alpha|$. Using
\[
\Gamma(d+1+i)=\Gamma(1+i)\prod_{k=1}^d (k+i)
\]
and $|\Gamma(1+i)|=|\Gamma(1-i)|=|\alpha|$, we obtain
\[
|\beta_d|^2=|\alpha|^4\prod_{k=1}^d\left(1+\frac1{k^2}\right).
\]
By Lemma \ref{lem:bergman-constants},
\[
|\alpha|^2=\frac{\pi}{\sinh\pi}.
\]
The Euler product for $\sinh$ gives
\[
\prod_{k=1}^\infty \left(1+\frac1{k^2}\right)=\frac{\sinh\pi}{\pi}.
\]
Therefore $|\beta_d|<|\alpha|$. It follows that
\[
c_d-|\beta_d|>0
\qquad\text{and}\qquad
c_d-|\alpha|<0.
\]
The first inequality gives the stated formula for $\liminf_{|z|\to 1^-}\widetilde f_d(z)$, and the second gives
\[
\liminf_{m\to\infty}\lambda_m=c_d-|\alpha|<0.
\]
Choosing integers $m_k\to\infty$ with $\lambda_{m_k}\to c_d-|\alpha|$, Weyl's criterion applied to $e_{(m_k,0,\dots,0)}$ shows that
\[
c_d-|\alpha|\in \sigma_{\mathrm{ess}}(T_{f_d}).
\]
Hence $T_{f_d}$ is not essentially positive.
\end{proof}

\begin{remark}
The fixed constant $\frac12$ still works on $A^2(\mathbb B_d)$ as long as $|\beta_d|<\frac12$. The product formula above shows that $|\beta_d|$ increases strictly with $d$. Exact computations show that $|\beta_{11}| < 0.5$ while $|\beta_{12}|>0.5$. 
Hence the symbol
\[
\frac12+\cos\left(\log \frac1{1-|z|^2}\right)
\]
is still a counterexample on $A^2(\mathbb B_d)$ for $1\le d\le 11$, but not for $d\ge 12$.
\end{remark}

\subsection*{Acknowledgements}
The author was supported by a Taussky--Todd Fellowship. The author thanks Jeck Lim for helpful comments on an earlier draft.

\bibliographystyle{plain} 
\bibliography{bibliography}

\end{document}